\documentclass{amsart}
\usepackage{amsfonts}

\newtheorem{theorem}{Theorem}
\theoremstyle{plain}

\newtheorem{definition}{Definition}

\newtheorem{remark}{Remark}

\numberwithin{equation}{section}
\input{tcilatex}

\begin{document}
\title[$P_{k}$ SETS ]{Extendibility of Some $P_{k}$ Sets}
\author{Bilge PEKER}
\address{Mathematics Education Programme, Ahmet Kele\c{s}o\u{g}lu Education
Faculty, Necmettin Erbakan University, Konya, Turkey.}
\email{bilge.peker@yahoo.com}
\author{Selin (INAG) CENBERCI}
\address{Mathematics Education Programme, Ahmet Kele\c{s}o\u{g}lu Education
Faculty, Necmettin Erbakan University, Konya, Turkey.}
\email{scenberci@konya.edu.tr}
\date{April 4, 2017}

\begin{abstract}
A set of \textit{m} distinct positive integers $\left\{
a_{1},...a_{m}\right\} $ is called a Diophantine m-tuple if $a_{i}a_{j}+n$
is a square for each $1\leq i<j\leq m$ . The aim of this study is to show
that some $P_{k}$ sets can not be extendible to a Diophantine quadruple when 
$k=2$ and $k=-3$ and also to give some properties about $P_{k}$ sets.
\end{abstract}

\maketitle

INTRODUCTION

\bigskip

Let $n$ be an integer and $A$ be a set $\{a_{1},a_{2},a_{3},...,a_{m}\}$
with $m$ different positive integers. This set has the property $D\left(
n\right) $ if $a_{i}a_{j}+n$ is a perfect square for all $1\leq i<j\leq m.$
Such a set is called as a Diophantine m-tuple (with the property $D\left(
n\right) $) or a $P_{n}$ set of size $m$.

The problem extending of $P_{k}$ sets is an old one dating from the times of
Diophantus $\left[ 3\right] .$ So, this problem has been studied
extensively. Diophantus studied finding four numbers such that the product
of any two of them, increased by 1, is a perfect square. In $\left[ 2,4,9%
\right] $, the more general form of this problem was considered.The most
famous result in this area is the one found by Baker and Davenport $\left[ 1%
\right] $ who proved for the four numbers $1,3,8,120$\ the property that the
product of any two, increased by $1$, is a perfect square. Dujella $\left[ 4%
\right] $ has a vast amount of literature about this interesting problem.

Filipin $\left[ 7\right] $ studied on the set $P_{-1}=\{1,10,c\}$ and proved
that there do not exist different positive integers $c,d>1$ such that the
product of any two distinct elements of the set $\{1,10,c,d\}$ diminished by 
$1$ is a perfect square.

In 1985, Brown $\left[ 2\right] $, Gupta and Singh $\left[ 8\right] $,
Mohanty and Ramasamy $\left[ 11\right] $ proved independently that if \ $%
n\equiv 2\left( \func{mod}4\right) $ then there does not exist a $P_{n}$ set
of size 4. Peker et al. $\left[ 16\right] $ showed that for an integer $k$
such that $\left\vert k\right\vert \geq 2,$ the $D\left( -2k+1\right) -$%
triple $\left\{ 1,k^{2},k^{2}+2k-1\right\} $ can not be extended to a $%
D\left( -2k+1\right) -$quadruple.

Mohanty and Ramasay $\left[ 10\right] $ showed non-extendibility of the set $%
P_{-1}=\{1,5,10\}$. Kedlaya $\left[ 9\right] $ gave a list of
non-extendibity of $P_{-1}$ $\ $triples $\{1,2,45\}$, $\{1,2,4901\}$, $%
\{1,5,65\}$, $\{1,5,20737\}$, $\{1,10,17\}$ and $\{1,26,37\}$. Both
Thamotherampillai $\left[ 17\right] $ and Mootha \& Berzsenyi $\left[ 13%
\right] $ proved the non-extendibility of the set $P_{2}=\{1,2,7\}$ by u$%
\sin $g different methods. Ozer $\left[ 15\right] $ gave some $P_{11}$ and $%
P_{-11\text{ }}$ triples.

Diophantine triples are classified as regular or irregular according to
whether they provide the condition below.

\begin{definition}
$\left[ 6\right] $ A $P_{k}$-triple $\left\{ a,b,c\right\} $ is called
regular if it satisfies the condition $\left( c-b-a\right) ^{2}=4\left(
ab+k\right) $. This equation is symmetric under permutations of $a,b,c$.
\end{definition}

\begin{definition}
\bigskip $\left[ 12\right] $ If $p$ is an odd prime and $\left( a,p\right)
=1 $ then $\left( \frac{a}{p}\right) $ is defined as Legendre Symbol and $%
\left( \frac{a}{p}\right) $=$\left\{ 
\begin{array}{c}
\overset{}{+1,}\text{ \ \ \ } \\ 
-1,\text{ \ \ \ \ }%
\end{array}%
\right. 
\begin{array}{c}
if \\ 
if%
\end{array}%
\begin{array}{c}
aRp \\ 
aNp%
\end{array}%
$

\begin{theorem}
$\left[ 12\right] $ (Euleur Criterion) If $p$ is a prime and $p$ does not
divide a positive integer $a$, then $\left( \frac{a}{p}\right) \equiv a^{%
\frac{p-1}{2}}\left( \func{mod}p\right) $ .
\end{theorem}
\end{definition}

In this paper, we give some$\ $non-extendable $P_{2}$ and $P_{-3}$ sets by
using solutions of simultaneous Pell equations.

\begin{theorem}
\textit{\ \ }The set $P_{2}=\{7,14,41\}$ is non-extendable.
\end{theorem}

\begin{proof}
\textit{Assume that the Diophantine triple }$P_{2}=\{7,14,41\}$\textit{\
with the property }$D\left( 2\right) $ \textit{can be extended with }$m$ to
a Diophantine quadruple.\textit{\ Then there exist integers }$x,y,z$\textit{%
\ such that \ }$\ \ \ \ \ \ \ \ \ \ $%
\begin{equation*}
\ \ 7m+2=x^{2}\ \ \ \ \ \ \ \ \ \left( 1\right) \mathit{\ \ \ }
\end{equation*}%
\textit{\ \ \ \ \ \ \ \ \ \ \ \ \ \ \ \ \ }$\ \ \ \ \ \ \ \ \ \ \ \ \ \ \ \
\ \ \ \ \ \ \ \ \ \ \ \ $%
\begin{equation*}
14m+2=y^{2}\ \ \ \ \ \ \ \ \ \left( 2\right) \mathit{\ \ }
\end{equation*}%
\textit{\ \ \ \ \ \ \ \ \ \ \ \ \ \ \ \ \ \ \ \ \ \ }$\ \ \ \ \ \ \ \ \ \ \
\ \ \ \ \ \ \ \ \ \ \ \ \ $%
\begin{equation*}
41m+2=z^{2}\ \ \ \ \ \ \ \ \ \left( 3\right) .\mathit{\ \ }
\end{equation*}%
\textit{\ \ \ \ \ By eliminating }$m$\textit{\ from the equations}$\ \left(
1\right) $\textit{\ and }$\left( 2\right) $,\textit{\ we get\ \ \ \ \ \ \ \
\ \ \ \ \ \ \ \ \ \ \ \ \ \ \ }$\mathit{\ \ }$%
\begin{equation*}
y^{2}-2x^{2}=-2\ \ 
\end{equation*}%
\textit{and this gives \ \ \ \ \ \ \ \ \ \ \ \ \ \ \ \ \ \ \ \ \ \ \ \ \ }$%
\mathit{\ }$%
\begin{equation*}
2(x^{2}-1)=y^{2}\text{ \ \ \ \ \ \ }\left( 4\right) .
\end{equation*}

\textit{In the last equation, since the left hand side is even, the right
hand side must be even, too. So, we can write \ }$y=2y_{1}$ $\left( y_{1}\in 
%TCIMACRO{\U{2124} }%
%BeginExpansion
\mathbb{Z}
%EndExpansion
\right) $. \textit{Then the last equation becomes\ \ \ \ \ \ \ \ \ \ \ \ \ \
\ \ \ \ \ \ \ \ \ \ \ }%
\begin{equation*}
\mathit{\ }x^{2}-1=2y_{1}^{2}.
\end{equation*}

\textit{From this, we can conclude that }$x$\textit{\ must be odd, that is, }%
$x=2x_{1}+1\ (x_{1}\in 
%TCIMACRO{\U{2124} }%
%BeginExpansion
\mathbb{Z}
%EndExpansion
)$. \textit{Then, the last equation will be as follows}\ \ \ \ \ \ \ \ \ \ \
\ \ \ \ \ \ \ \ \ \ \ \ \ 
\begin{equation*}
2x_{1}\left( x_{1}+1\right) =y_{1}^{2}.
\end{equation*}%
Since \textit{the left hand side of this equation has multiple of }$2$,%
\textit{\ the right hand side must be even. So, we can write }$y_{1}=2y_{2}\
\left( y_{2}\in 
%TCIMACRO{\U{2124} }%
%BeginExpansion
\mathbb{Z}
%EndExpansion
\right) .$ Now, we have\textit{\ }$y=4y_{2}.\ $\textit{If we write }$%
y=4y_{2} $\textit{\ in the equation }$(4),$\textit{\ we find}\ \ \ \ \ \ \ \
\ \ \ \ \ \ \ \ \ \ \ \ \ 
\begin{equation*}
x^{2}-8y_{2}^{2}=1.
\end{equation*}%
\textit{The fundamental solution of this Pell equation\ is }$\left(
x,y_{2}\right) =(3,1)\in 
%TCIMACRO{\U{2124} }%
%BeginExpansion
\mathbb{Z}
%EndExpansion
^{2}.$ \textit{So, all solutions of this Pell equation can be given in the
form }$x_{n}+\sqrt{8}(y_{2})_{n}=(3+\sqrt{8})^{n}$ \textit{by the usual
methods. We get the general recurrence relation for the solutions of }$x$%
\textit{\ as} $x_{n}=6x_{n-1}-x_{n-2}.$\textit{\ Using this general
recurrence relation, we can get some values of }$m$\textit{\ from the
equation }$\left( 1\right) $\textit{. For these values of }$m$,\textit{\ one
can easily check that none of these values give any square of an integer for
the equation }$\left( 3\right) $\textit{. So, we can not get an integer
solution for }$z.$ \textit{This means that the set }$P_{2}=\{7,14,41\}$%
\textit{\ is non-extendable.}
\end{proof}

\begin{remark}
\bigskip A $P_{2}$-triple $\left\{ 7,14,41\right\} $ is called regular.
\end{remark}

Now, in the following theorem we show that another $P_{2}$ triple, namely $%
\{1,7,14\}$ can not be extended to a Diophantine quadruple by using above
method.

\begin{theorem}
The set $P_{2}=\{1,7,14\}$ is non-extendable.
\end{theorem}

\begin{proof}
\textit{Assume that there is a positive integer }$m$\textit{\ and the set }$%
P_{2}=\{1,7,14\}$\textit{\ can be extended with }$m.$\textit{\ Then, let us
find an integer }$m$\textit{\ such that, \ }$\ \ \ \ \ \ \ \ \ \ $%
\begin{equation*}
\ \ m+2=x^{2}\ \ \ \ \ \ \ \ \ \left( 5\right) \mathit{\ \ \ }
\end{equation*}%
\textit{\ \ \ \ \ \ \ \ \ \ \ \ \ \ \ \ \ }$\ \ \ \ \ \ \ \ \ \ \ \ \ \ \ \
\ \ \ \ \ \ \ \ \ \ \ \ $%
\begin{equation*}
7m+2=y^{2}\ \ \ \ \ \ \ \ \ \left( 6\right) \mathit{\ \ }
\end{equation*}%
\textit{\ \ \ \ \ \ \ \ \ \ \ \ \ \ \ \ \ \ \ \ \ \ }$\ \ \ \ \ \ \ \ \ \ \
\ \ \ \ \ \ \ \ \ \ \ \ \ $%
\begin{equation*}
14m+2=z^{2}\ \ \ \ \ \ \ \ \ \left( 7\right) .\mathit{\ \ }
\end{equation*}%
where $x,y,z$ are some integers satisfying the above equations.\textit{\ \ \
\ \ \ }

\textit{By eliminating }$m$\textit{\ from\ the equations }$\left( 6\right) $%
\textit{\ and }$\left( 7\right) $ \textit{we get}

\begin{equation*}
z^{2}-2y^{2}=-2\ \ \ \ 
\end{equation*}%
\textit{and this yields \ \ \ \ \ \ \ \ \ \ \ \ \ \ \ \ \ \ \ \ \ \ \ \ \ }$%
\mathit{\ }$%
\begin{equation*}
2(y^{2}-1)=z^{2}\text{ \ \ \ \ \ \ }\left( 8\right) .
\end{equation*}

\textit{In the last equation, since the left side is even, the right side
must be even, too. So, we can write }$\mathit{z}=2z_{1}$ $\left( z_{1}\in 
%TCIMACRO{\U{2124} }%
%BeginExpansion
\mathbb{Z}
%EndExpansion
\right) ,$\textit{\ then the equation }$\left( 8\right) $ \textit{becomes \
\ \ \ \ \ \ \ \ \ \ \ \ \ \ \ \ \ \ \ \ \ \ \ \ }%
\begin{equation*}
\mathit{\ }y^{2}-1=2z_{1}^{2}.
\end{equation*}

\textit{From this, we can conclude that }$y$\textit{\ must be odd, that is, }%
$y=2y_{1}+1\ (y_{1}\in 
%TCIMACRO{\U{2124} }%
%BeginExpansion
\mathbb{Z}
%EndExpansion
)$. \textit{Then, we find}\ \ \ \ \ \ \ \ \ \ \ \ \ \ \ \ \ \ \ \ \ \ \ \ 
\begin{equation*}
2y_{1}\left( y_{1}+1\right) =z_{1}^{2}.
\end{equation*}%
Since the \textit{left hand side of this equation has multiple of }$2$,%
\textit{\ the right hand side must be even. So, we can write }$z_{1}=2z_{2}\
\left( z_{2}\in 
%TCIMACRO{\U{2124} }%
%BeginExpansion
\mathbb{Z}
%EndExpansion
\right) .$

\textit{Consequently, }$z=4z_{2}.\ $\textit{If we write }$z=4z_{2}$\textit{\
in the equation }$(8),$\textit{\ we find}\ \ \ \ \ \ \ \ \ \ \ \ \ \ \ \ \ \
\ \ \ 
\begin{equation*}
y^{2}-8z_{2}^{2}=1.
\end{equation*}%
\textit{The fundamental solution of this Pell equation\ is }$\left(
y,z_{2}\right) =(3,1)\in 
%TCIMACRO{\U{2124} }%
%BeginExpansion
\mathbb{Z}
%EndExpansion
^{2}.$ \textit{So, all solutions of this Pell equation can be given in the
form }$y_{n}+\sqrt{8}(z_{2})_{n}=(3+\sqrt{8})^{n}$ \textit{by the usual
methods. We get the general recurrence relation for the solutions of }$y$%
\textit{\ as }$y_{n}=6y_{n-1}-y_{n-2}.$\textit{\ Using this general
recurrence relation, we can get some values of }$m$\textit{\ from the
equation }$\left( 6\right) $\textit{. For these values of }$m$,\textit{\ one
can easily check that none of these values give any square of an integer for
the equation }$\left( 5\right) $\textit{. So, we can not get an integer
solution for }$x.$ \textit{This means that the set }$P_{2}=\{1,7,14\}$%
\textit{\ is non-extendable.}\ 
\end{proof}

\begin{remark}
\bigskip \bigskip A $P_{2}$-triple $\left\{ 1,7,14\right\} $ is called
regular.
\end{remark}

In addition, we can give another $P_{2}$ triple which can not be extended to
a quadruple.

\begin{theorem}
\ \ The set $P_{2}=\{41,239,478\}$ is non-extendable.\ \ \ \ \ 
\end{theorem}

\begin{proof}
\textit{Assume that there is a positive integer }$m$\textit{\ and the set }$%
P_{2}=\{41,239,478\}$\textit{\ can be extended with }$m.$\textit{\ Then, let
us find an integer }$m$\ such that,\textit{\ \ }$\ \ \ \ \ \ \ \ \ \ $%
\begin{equation*}
\ \ 41m+2=x^{2}\ \ \ \ \ \ \ \ \ \left( 9\right) \mathit{\ \ \ }
\end{equation*}%
\textit{\ \ \ \ \ \ \ \ \ \ \ \ \ \ \ \ \ }$\ \ \ \ \ \ \ \ \ \ \ \ \ \ \ \
\ \ \ \ \ \ \ \ \ \ \ \ $%
\begin{equation*}
239m+2=y^{2}\ \ \ \ \ \ \ \ \ \left( 10\right) \mathit{\ \ }
\end{equation*}%
\textit{\ \ \ \ \ \ \ \ \ \ \ \ \ \ \ \ \ \ \ \ \ \ }$\ \ \ \ \ \ \ \ \ \ \
\ \ \ \ \ \ \ \ \ \ \ \ \ $%
\begin{equation*}
478m+2=z^{2}\ \ \ \ \ \ \ \ \left( 11\right) .\mathit{\ \ }
\end{equation*}%
where $x,y,z$ are some integers satisfying the above equations.\textit{\ \ \
\ \ \ }

\textit{By eliminating }$m$\textit{\ from the equations }$\left( 10\right) $%
\textit{\ and }$\left( 11\right) $,\textit{\ we get}

\begin{equation*}
2y^{2}-z^{2}=2\ \ \ 
\end{equation*}%
\textit{and \ \ \ \ \ \ \ \ \ \ \ \ \ \ \ \ \ \ \ \ \ \ \ \ \ }$\mathit{\ }$%
\begin{equation*}
2(y^{2}-1)=z^{2}.
\end{equation*}

\textit{Since this equation is the same, equation }$\left( 4\right) $\textit{%
\ and }$\left( 8\right) ,$\textit{\ in the next steps, we will apply similar
methods done in Theorem 1 and 2. Hence, this completes the proof.}
\end{proof}

\begin{remark}
\textit{\ }\bigskip A $P_{2}$-triple $\left\{ 41,239,478\right\} $ is called
regular.\textit{\ \ }\ \ \ \ \ \ 
\end{remark}

\begin{remark}
\textit{\ }$\ $There is no set $P_{2}$ including any positive multiple of $3$%
.
\end{remark}

\bigskip $\ $\textit{Let's assume that }$3k$\textit{\ }$(k\in 
%TCIMACRO{\U{2124} }%
%BeginExpansion
\mathbb{Z}
%EndExpansion
^{+})$\textit{\ is an element of the set }$P_{2}.$\textit{\ For any }$%
t\epsilon P_{2}$,\textit{\ the equation\ \ \ \ \ \ \ \ \ \ \ \ \ \ \ \ \ \ \
\ }%
\begin{equation*}
3kt+2=x^{2}
\end{equation*}%
\textit{must be satisfied. In modulo3},\textit{\ the following equation is
deduced,\ \ \ \ \ \ \ \ \ \ \ \ \ \ }

\textit{\ \ \ \ \ \ }%
\begin{equation*}
\mathit{\ \ \ \ }x^{2}\equiv 2\ (\func{mod}3).
\end{equation*}

\textit{Since }$3$\textit{\ is an odd prime and }$\left( 2,3\right) =1,$%
\textit{\ using Legendre Symbol and Euleur Criterion, we get }$\left( \dfrac{%
2}{3}\right) =-1$\textit{\ which means that the last equation \ is
unsolvable. Therefore, }$3k$\textit{\ }$(k\in 
%TCIMACRO{\U{2124} }%
%BeginExpansion
\mathbb{Z}
%EndExpansion
^{+})$\textit{\ can not be an element of the\ set }$P_{2}.$

\begin{theorem}
The $P_{-3}$ set $\left\{ 3,4,13\right\} $ can not be extendible.
\end{theorem}

\begin{proof}
Assume that there is a positive integer $m$\ and the set $P_{-3}=\left\{
3,4,13\right\} $\ can be extended with $m.$\ Then let us find integers $%
x,y,z $\ such that,$\ \ \ \ \ \ \ \ $

\begin{equation*}
\ \ 3m-3=x^{2}\ \ \ \ \ \ \ \ \left( 12\right) \mathit{\ \ \ }
\end{equation*}%
\textit{\ \ \ \ \ \ \ \ \ \ \ \ \ \ \ \ \ }%
\begin{equation*}
\ \ 4m-3=y^{2}\ \ \ \ \ \ \ \ \left( 13\right) \mathit{\ \ \ }
\end{equation*}%
$\ \ \ \ \ \ \ \ \ \ \ \ \ \ \ \ \ \ \ \ \ \ \ \ \ \ \ \ \ \ \ \ \ \ \ \ $

$\ $%
\begin{equation*}
\ \ 13m-3=z^{2}\ \ \ \ \ \ \ \left( 14\right) \mathit{\ \ \ }
\end{equation*}%
where $x,y,z$ are some integers satisfying the above equations. Eliminating
of $m$ between $\left( 12\right) $, $\left( 13\right) $ yields

\begin{equation*}
\ \ 3y^{2}-4x^{2}=3\ \ \ \ \ \ \ \left( 15\right) \mathit{\ \ \ }
\end{equation*}

and then we have to abtain following equation

\begin{equation*}
\ \ 3(y^{2}-1)=4x^{2}\ \ \ \ \ \ \left( 16\right) \mathit{\ \ \ }
\end{equation*}

If we consider equation $\left( 16\right) $, since left hand side is divided
3, right hand side must be divided 3 too. So $3\mid x^{2}$ and consequently $%
3\mid x.$ And we write $x=3x_{1}$ $\left( x_{1}\in 
%TCIMACRO{\U{2124} }%
%BeginExpansion
\mathbb{Z}
%EndExpansion
\right) $. Byusing this relation if we rewrite equation $\left( 16\right) ,$%
we get

\begin{equation*}
\ \ y^{2}-1=12x_{1}^{2}\ \ \ \ \ \ \left( 17\right) \mathit{\ \ }
\end{equation*}

this equation gives us $y$ must be odd. Therefore $y=2y_{1}+1$ $\left(
y_{1}\in 
%TCIMACRO{\U{2124} }%
%BeginExpansion
\mathbb{Z}
%EndExpansion
\right) $ and from last equation \ we obtain the last equation,

\begin{equation*}
\ y_{1}\left( y_{1}+1\right) =3x_{1}^{2}.\ \ \ \ \ \ \left( 18\right) 
\mathit{\ \ \ }
\end{equation*}

Left hand side of the equation \ $\left( 18\right) $ is the multiplication
of two consecutive numbers. And we know that multiplication of two
consecutive numbers is always even. So the right hand side must be even,
that is\textit{\ }$x_{1}=2x_{2}\ \left( x_{2}\in 
%TCIMACRO{\U{2124} }%
%BeginExpansion
\mathbb{Z}
%EndExpansion
\right) $\textit{. }We obtain from this\textit{\ }$x=6x_{2}\ $\textit{, }and
the equation $(18)$\textit{\ }\ turn such Pell equation

\begin{equation*}
\ y^{2}-48x_{2}^{2}=1.\ \ \ \ \ \ \left( 19\right) \mathit{\ \ \ }
\end{equation*}

This Pell equations' main solution is $\left( x_{2},y\right) =(1,7)\in 
%TCIMACRO{\U{2124} }%
%BeginExpansion
\mathbb{Z}
%EndExpansion
^{2}.$\textit{\ }All of the solutions of this Pell equation can be find such
that\textit{\ }$y_{n}+\sqrt{48}\left( x_{2}\right) _{n}=\left( 7+\sqrt{48}%
\right) ^{n},$\ $n=0,1,2,...$\ .For \textit{\ }$y$ solutions, the \ general
recurrence relation is \textit{\ }$y_{n+2}=14y_{n+1}-y_{n}.$\textit{\ }For%
\textit{\ }considering equations\textit{\ }$\left( 12\right) ,\left(
13\right) ,\left( 14\right) $ we can not find a $m$ value, this gives our
assumption is not true. \ So the $P_{-3}$ set $\left\{ 3,4,13\right\} $ can
not be extended. This completes the proof.
\end{proof}

\begin{theorem}
The set $P_{2}=\{7,41,82\}$ can not be extendible.
\end{theorem}

\begin{proof}
\textit{Assume that there is a positive integer }$m$\textit{\ and the set }$%
P_{2}=\{7,41,82\}$\textit{\ can be extended with }$m.$\textit{\ Then let us
find integers }$x,y,z$\textit{\ such that \ }$\ \ \ \ \ \ \ \ \ \ $%
\begin{equation*}
\ \ 7m+2=x^{2}\ \ \ \ \ \ \ \ \ \left( 20\right) \mathit{\ \ \ }
\end{equation*}%
\textit{\ \ \ \ \ \ \ \ \ \ \ \ \ \ \ \ \ }$\ \ \ \ \ \ \ \ \ \ \ \ \ \ \ \
\ \ \ \ \ \ \ \ \ \ \ \ $%
\begin{equation*}
41m+2=y^{2}\ \ \ \ \ \ \ \ \ \left( 21\right) \mathit{\ \ }
\end{equation*}%
\textit{\ \ \ \ \ \ \ \ \ \ \ \ \ \ \ \ \ \ \ \ \ \ }$\ \ \ \ \ \ \ \ \ \ \
\ \ \ \ \ \ \ \ \ \ \ \ \ $%
\begin{equation*}
82m+2=z^{2}\ \ \ \ \ \ \ \ \left( 22\right) .\mathit{\ \ }
\end{equation*}%
\textit{\ \ \ \ \ \ By eliminating }$m$\textit{\ from the }$\ $\textit{%
equations }$\left( 21\right) $\textit{\ and }$\ \left( 22\right) $\textit{\
we get}

\begin{equation*}
2y^{2}-z^{2}=2\ \ \ 
\end{equation*}

\textit{and \ \ \ \ \ \ \ \ \ \ \ \ \ \ \ \ \ \ \ \ \ \ \ \ \ }$\mathit{\ }$%
\begin{equation*}
2(y^{2}-1)=z^{2}.
\end{equation*}

\textit{Since this equation is the same equation with }$\left( 4\right) $%
\textit{\ and }$\left( 8\right) ,$\textit{\ the next steps are same in
Theorem 1 and 2 . Therefore this completes the proof.}
\end{proof}

\begin{theorem}
\textit{\ }The set $P_{2}=\{41,82,239\}$ can not be extendible.
\end{theorem}

\begin{proof}
\textit{Assume that there is a positive integer }$m$ an the set \textit{\ }$%
P_{2}=\{41,82,239,m\}$\textit{\ can be extendible with }$m$. \textit{Then
let us find an integer }$m$\textit{\ such that}$\ \ \ \ \ \ \ \ \ $%
\begin{equation*}
\ \ 41m+2=x^{2}\ \ \ \ \ \ \ \ \ \left( 23\right) \mathit{\ \ \ }
\end{equation*}%
\textit{\ \ \ \ \ \ \ \ \ \ \ \ \ \ \ \ \ }$\ \ \ \ \ \ \ \ \ \ \ \ \ \ \ \
\ \ \ \ \ \ \ \ \ \ \ \ $%
\begin{equation*}
82m+2=y^{2}\ \ \ \ \ \ \ \ \ \left( 24\right) \mathit{\ \ }
\end{equation*}%
\textit{\ \ \ \ \ \ \ \ \ \ \ \ \ \ \ \ \ \ \ \ \ \ }$\ \ \ \ \ \ \ \ \ \ \
\ \ \ \ \ \ \ \ \ \ \ \ \ $%
\begin{equation*}
239m+2=z^{2}\ \ \ \ \ \ \ \ \ \left( 25\right) .\mathit{\ \ }
\end{equation*}%
\textit{\ \ \ \ \ where }$x,y,z$ \textit{are some integers satisfying above
equations. Eliminating }$m$\textit{\ from the equations}$\ \left( 23\right) $%
\textit{\ and }$\left( 24\right) $\textit{\ yieldst\ \ \ \ \ \ \ \ \ \ \ \ \
\ \ \ \ \ \ \ \ \ \ }$\mathit{\ \ }$%
\begin{equation*}
y^{2}-2x^{2}=-2\ \ 
\end{equation*}%
\textit{and then we obtain following equation\ \ \ \ \ \ \ \ \ \ \ \ \ \ \ \
\ \ \ \ \ \ \ \ }$\mathit{\ }$%
\begin{equation*}
2(x^{2}-1)=y^{2}\text{ \ \ \ \ \ \ }\left( 26\right) .
\end{equation*}

\textit{In the last equation, since the left side is even, the right side
must be even too. So we can write \ }$y=2y_{1}$ $\left( y_{1}\in 
%TCIMACRO{\U{2124} }%
%BeginExpansion
\mathbb{Z}
%EndExpansion
\right) ,$\textit{\ last equation gives us\ \ \ \ \ \ \ \ \ \ \ \ \ \ \ \ \
\ \ \ \ \ \ \ \ }%
\begin{equation*}
\mathit{\ }x^{2}-1=2y_{1}^{2}.
\end{equation*}

\textit{From this, we can conclude that }$x$\textit{\ must be odd. If we put 
}$x=2x_{1}+1\ (x_{1}\in 
%TCIMACRO{\U{2124} }%
%BeginExpansion
\mathbb{Z}
%EndExpansion
),$\textit{\ then we find}\ \ \ \ \ \ \ \ \ \ \ \ \ \ \ \ \ \ \ \ \ \ \ \ 
\begin{equation*}
2x_{1}\left( x_{1}+1\right) =y_{1}^{2}.
\end{equation*}%
\textit{The left hand side of this equation has multiple of }$2,$\textit{\
that is, the right hand side must be even. So we can put }$y_{1}=2y_{2}\
\left( y_{2}\in 
%TCIMACRO{\U{2124} }%
%BeginExpansion
\mathbb{Z}
%EndExpansion
\right) .$ \textit{Thus we can write }$y=4y_{2}.\ $\textit{Putting }$y=4y_{2}
$\textit{\ in the equation }$(26)\ \ $\textit{yields}\ \ \ \ \ \ \ \ \ \ \ \
\ \ \ \ \ \ \ 
\begin{equation*}
x^{2}-8y_{2}^{2}=1.
\end{equation*}%
\textit{Main solution of this Pell equation\ is }$\left( x,y_{2}\right)
=(1,7)\in 
%TCIMACRO{\U{2124} }%
%BeginExpansion
\mathbb{Z}
%EndExpansion
^{2}.$ \textit{All \ of the solutions of this Pell equation are of the form }%
$x_{n}+\sqrt{8}(y_{2})_{n}=(3+\sqrt{8})^{n}$, $n=0,1,2,3,...$\textit{. For }$%
x$ \textit{\ solutions } \textit{general recurrence relation is }$%
x_{n}=6x_{n-1}-x_{n-2}.$\textit{\ From the equations }$\left( 23\right)
,\left( 24\right) ,\left( 25\right) $\textit{\ we can not find a }$m$\textit{%
\ value. Hence, our assumption is not true. This completes the proof.}
\end{proof}

\begin{remark}
There is no set not only $P_{2}$ but also $P_{-3}$ including any positive
multiple of 5.
\end{remark}

Assume that $5k$ $\left( k\in 
%TCIMACRO{\U{2124} }%
%BeginExpansion
\mathbb{Z}
%EndExpansion
^{+}\right) $ is an element of both the sets $P_{2}$ and $P_{-3}$. For any $%
t_{1}\in P_{2}$ and $t_{2}\in P_{-3},$ we can write the following equations,

\begin{equation*}
5kt_{1}+2=x^{2}\text{ \ \ \ \ and \ \ \ \ }\ 5kt_{2}-3=y^{2}.
\end{equation*}

From these equations, one can deduce in modulo 5 the followings

\begin{equation*}
x^{2}\equiv 2\left( \func{mod}5\right) \text{ \ \ \ \ and \ \ \ \ }\
y^{2}\equiv 2\left( \func{mod}5\right) .\text{ \ \ \ \ \ \ }\left( 27\right) 
\text{\ }
\end{equation*}

Since $5$ is an odd prime and $\left( 2,5\right) =1$, using from Legendre
Symbol, we get $\left( \frac{2}{5}\right) =-1$ which means that the equation 
$\left( 27\right) $ is unsolvable. Therefore, $5k$ $\left( k\in 
%TCIMACRO{\U{2124} }%
%BeginExpansion
\mathbb{Z}
%EndExpansion
^{+}\right) $ can not be an element of both the sets $P_{2}$ and $P_{-3}$.

\textit{\ \ \ }

\end{document}